\documentclass[12pt]{amsart}
\usepackage{amssymb}
\usepackage{ulem}
\usepackage{amsmath,amscd}
\usepackage[utf8]{inputenc}
\usepackage{graphicx}
\usepackage{stmaryrd}
\usepackage{mathtools}
\usepackage{comment}
\usepackage{multicol}
\usepackage{xcolor}
\usepackage{tikz-cd}
\usepackage{tikz}

\newcommand{\ggadd}[1]{{\color{red}\marginpar{\color{red}[GG]}#1}}
\newcommand{\aadd}[1]{{\color{green}\marginpar{\color{green}[AŻ]}#1}}

\newcommand{\pcsadd}[1]{{\color{blue}\marginpar{\color{blue}[PCS]}#1}}

\newtheorem{theorem}{Theorem}[section]
\newtheorem{lemma}[theorem]{Lemma}
\newtheorem{example}[theorem]{Example}
\newtheorem{problem}[theorem]{Problem}
\newtheorem{corollary}[theorem]{Corollary}

\newtheorem{proposition}[theorem]{Proposition}

\theoremstyle{definition}
\newtheorem{definition}[theorem]{Definition}
\theoremstyle{remark}
\newtheorem{remark}[theorem]{Remark}

\numberwithin{equation}{section}

\newcommand{\Ind}{{\rm Ind}}
\newcommand{\inter}{{\rm i}}

\newcommand{\Image}{{\rm Im}}

\newcommand{\Coin}{{\rm Coin}}
\newcommand{\MC}{{\rm MC}}

\newcommand{\R}{{\mathbb{R}}}
\newcommand{\modulo}{{\textrm{mod}~}}

\newcommand{\GCD}{{\rm GCD}}
\newcommand{\Int}{{\rm Int}}

\newcommand{\pr}{{\rm pr}}
\begin{document}

\vspace{0.5in}

\vspace{0.5in}

\title[Nielsen coincidence theory of $(n,m)$-valued pairs of maps] {Nielsen coincidence theory of $(n,m)$-valued pairs of maps}


\author[Graff]{Grzegorz Graff}
\address{Faculty of Applied Physics and Mathematics, Gda\'nsk University of
Technology, Narutowicza 11/12, 80-233 Gda\'nsk, Poland,
Orcid: 0000-0001-5670-5729}

\email{grzegorz.graff@pg.edu.pl}

\author[Staecker]{P. Christopher Staecker}
\address{Mathematics Department, Fairfield University, 1073 North Benson Rd, Fairfield, 06823-519, CT, USA, Orcid: 0000-0002-4182-7122}
\email{cstaecker@fairfield.edu}

\author[\.Zeromski]{Alan \.Zeromski}
\address{Doctoral School of Gda\'nsk University of Technology, Narutowicza 11/12, 80-233 Gda\'nsk, Poland,
Orcid: 0009-0001-3481-2897}
\email{alan.zeromski@pg.edu.pl}

\subjclass[MSC Classification]{55M20, 54C60.}
\keywords{Nielsen number, $n$-valued map, coincidence points, fixed
point index.}

\begin{abstract} 
We consider  pairs of maps 
$(f,g)$, where $f$ is an $n$-valued map and 
 $g$ is an $m$-valued map, defined on connected finite polyhedra. 
 A point $x$ such that $f(x)\cap g(x)\neq \emptyset$ is called a coincidence point of $f$ and $g$.
 
A useful device for studying coincidence points would be a Nielsen-type invariant which provides  a lower bound for the number of coincidence points of all $(n, m)$-valued pairs of maps homotopic to $(f,g)$. The construction of such an invariant $N(f:g)$ was proposed in [J. Fixed Point Theory Appl. 14, 309--324 (2013)]. Building on this approach, we observe that the definition of the coincidence set used there requires a correction in order to properly account for coincidences arising as intersections of the graphs of the two maps.

We therefore introduce a graph-based formulation of coincidence classes and define the corresponding geometric and algebraic Nielsen coincidence numbers.
In the case of $(n, m)$-valued pairs of maps of the circle our invariant provides a sharp lower
bound, which we precisely determine.
\end{abstract}

\maketitle

\section{Introduction}

Given a space $X$ and a function $f \colon X\to X$, a point $x\in X$ is called a \textit{fixed point} of $f$ if $x= f(x)$. 
Nielsen fixed point theory aims to measure the cardinality of the set of fixed points in a way that is invariant under homotopy of  $f$. Specifically, the Nielsen number $N(f)$ is a homotopy invariant and a lower bound  for the number of fixed points in the homotopy class of $f$ (cf. \cite{Jiang} for the standard reference). 

We will consider two well-studied generalizations of Nielsen fixed point theory: $n$-valued maps, and coincidence theory.

Given spaces $X$ and $Y$ and $n \in \mathbb{N}$, a multifunction
$f \colon X \multimap Y$ is called \textit{$n$-valued} if, for every $x \in X$,
the image $f(x) \subset Y$ is an unordered set of cardinality exactly $n$.
When $Y = X$ and $f$ is continuous (as defined in more detail below), we call
$f$ an \textit{$n$-valued self-map}, and a point $x \in X$ is called a
\textit{fixed point} of $f$ if $x \in f(x)$.

Topological fixed point theory of $n$-valued maps was first studied by Schirmer in \cite{Schirmer1}, with the goal of generalizing classical constructions from the single-valued case. Schirmer defined a fixed point index for $n$-valued maps, and proved a Nielsen fixed point theorem. 
The subject was developed further in several papers by R.~F.~Brown, beginning with \cite{Brown3}. 
In particular, it is shown there that $n$-valued maps of the circle can be classified up to homotopy by their ``degree'', and a convenient formula for the Nielsen number is obtained: if $f$ is an $n$-valued self-map of the circle of degree $a$, then
$N(f)=|a-n|.$
This generalizes the classical single-valued formula $N(f)=|a-1|$; see, for instance, \cite{Jez-Mar}. The theory of $n$-valued maps remains an active area of research and continues to attract attention, with new results and further developments  appearing in the literature (cf. \cite{dekimpe1, dekimpe2}).

Single-valued fixed point theory has also been generalized to \textit{coincidence theory} (cf. the survey paper \cite{Goncalves}, also for non-orientable manifolds see \cite{Jezierski}). For a pair of maps $f,g \colon X\to Y$, a point $x\in X$ is called a \textit{coincidence point} of the pair $(f,g)$ if $f(x)=g(x)$. Again the basics of this theory were laid out by Schirmer, who defined a coincidence point index and proved a Nielsen coincidence theorem in her dissertation \cite{Schirmer0}. The standard examples considered within this theory involve maps of circles, for which the following statement holds: if $f,g \colon S^1\to S^1$ are two maps of degrees $a$ and $b$, then $N(f:g) = |a-b|$. 

This paper aims to unify the two generalizations outlined above (mainly for $S^1$ but some facts are valid in more general settings). Let $f,g \colon S^1\multimap S^1$ be a pair of multimaps of the circle, and define the coincidence set of the pair $(f,g)$ as $\Coin(f:g) = \{x;~~ f(x) \cap g(x)\neq \emptyset \}$. 

One of our goals is to define the counterpart of the Nielsen number for the pair $(f,g)$, and show that it satisfies a formula which generalizes both the $n$-valued fixed point formula, and the single-valued coincidence formula. 

Brown \& Kolahi considered exactly the same situation in \cite{Brown1}, which is the main reference for this paper. The main result of \cite{Brown1}, Theorem 5.1, claimed the following formula. Let  $f$ be an $n$-valued map of degree $a$, and $g$ be an $m$-valued map of degree $b$ then the Nielsen coincidence number of the pair $(f,g)$ is given by
\begin{equation}\label{browneq}
\frac{|am-bn|}{\GCD(n,m)}. 
\end{equation}
Unfortunately there are several oversights and  errors in that paper which we will correct by modifying the construction proposed by Brown \& Kolahi.

\begin{figure}[h!]
\[ \includegraphics[width=0.75\textwidth]{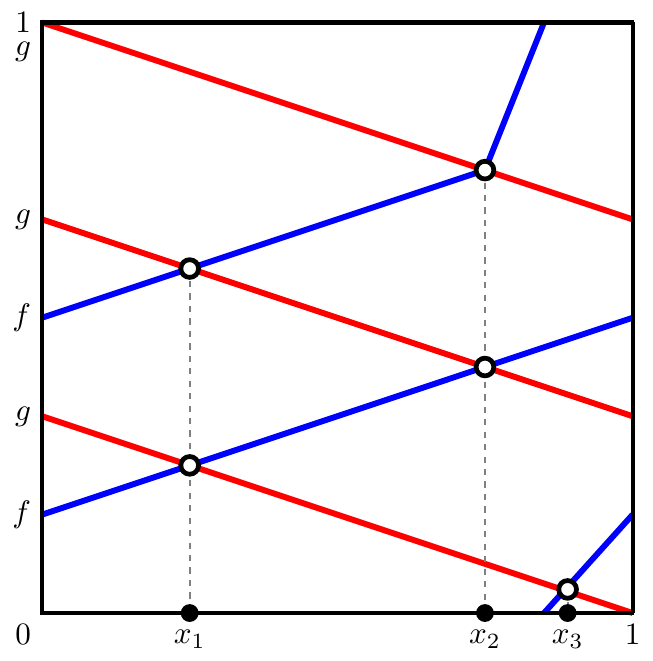} \]
%
%
%
%
%
%
%
%
\label{introfig}
\caption{Two multimaps of the circle, having 5 graph intersection points (circled), but only 3 domain coincidence points.}
\end{figure}

To briefly outline the problems in \cite{Brown1}, consider Figure \ref{introfig}, which shows two multimaps of the circle, $f$ and $g$. (Throughout the paper, we will represent the circle pictorially as the interval $[0,1]$ with endpoints identified.) In Figure \ref{introfig}, the map $f$ is $2$-valued with degree 1, and $g$ is $3$-valued with degree $-1$.  

In Figure~\ref{introfig} we see that $f$ and $g$ have three coincidence points, namely $x_1,x_2,x_3$.
However, the formula~\eqref{browneq} yields $N(f:g)=\frac{|1\cdot 3-(-1)\cdot 2|}{1}=5$, so $N(f:g)$ exceeds the actual number of coincidence points. This is a counterexample to Theorem 2.1 of \cite{Brown1}, which claims that $N(f:g)$ is a lower bound for the number of coincidence points.

In this paper, we will show that many of the results of \cite{Brown1} are correct if they are re-cast in terms of the set of \textit{graph intersection points} of the pair $(f,g)$, which are elements of $X\times Y$, rather than the set of coincidence points, which are elements of $X$. 
For example in Figure \ref{introfig}, the graph intersection points of $f$ and $g$ are the circled points of $S^1\times S^1$. There are 5 such points, which agrees with the expected count from \eqref{browneq}.

The contribution of the present paper is threefold. First, we show that the Nielsen-type construction proposed in~\cite{Brown1} cannot be used directly on the set of domain coincidence points, because in that setting the resulting relation does not correctly detect the homotopy behavior of coincidences. Second, we replace domain coincidence points by graph intersection points in $X \times Y$, which leads to a corrected equivalence relation and a well-defined Nielsen-type invariant. Third, for $(n,m)$-valued self-maps of the circle we reinterpret graphs as loops in the torus and use intersection-theoretic methods to compute the corresponding Nielsen number and the minimum number of graph intersection points in a homotopy class. 

The structure of this paper is as follows. In Section \ref{sec:nv} we recall basic facts related to multimaps, especially in context of self-multimaps of the circle. Section \ref{sec:power} is devoted to the study of $(n,m)$-valued pairs of power maps on the circle, which are the model case for our investigations. 


Next, in Section~4 we explain why the approach of~\cite{Brown1}, formulated in terms of domain coincidence points, is not adequate in our setting, and we replace it by a relation defined on graph intersection classes. This leads to the geometric Nielsen coincidence number $\hat N(f\colon g)$, which is a homotopy invariant and provides a lower bound for the minimal number of graph intersection points in the homotopy class of $(f,g)$. Although this gives the correct geometric framework for our problem, determining $\hat N(f:g)$ explicitly remains difficult, since it requires deciding which graph intersection classes are geometrically essential. This difficulty motivates the algebraic approach developed in the later sections.

In Section \ref{sec:intersection} we show the equivalence of multimaps of $S^1$ with torus loops. In this interpretation the coincidence points of multimaps become intersection points of images of the appropriate loops. This approach enables us to  apply the intersection number in our investigations and it seems that it may be of independent interest.

Finally, in Sections~6 and~7, we pass to the algebraic Nielsen coincidence number $N(f:g)$ and prove the main result of the paper (Theorem~7.1) by means of intersection theory. This algebraic viewpoint is closest to the original aim of Brown and Kolahi in~\cite{Brown1}, namely, to obtain a Nielsen-type invariant detecting the minimal coincidence behavior in a homotopy class. In our corrected setting, formulated in terms of graph intersection classes, we show that for an $(n,m)$-valued pair of self-maps of $S^1$ of degrees $a$ and $b$, the minimal number of graph intersection points is equal to $|am-bn|$.


\section{$n$-valued maps} \label{sec:nv}
In this section we give basic definitions and review some fundamental facts related to the class of $n$-valued maps.

Given sets $X, Y$ and $n \in \mathbb{N}$, a multifunction $f \colon X \multimap Y$ is \textit{$n$-valued}, if for every $x \in X$ the image $f(x)$ is an unordered set of cardinality $n$. 
An $n$-valued function $f \colon X \multimap Y$ is called \textit{upper semicontinuous} if, for each open set $U \subseteq Y$, the set $\{ x \in X;~~ f(x) \subseteq U \}$ is open in $X$ and it is \textit{lower semicontinuous} if $\{ x \in X;~~ f(x) \cap U  \neq \emptyset \}$ is open in $X$. An $n$-valued function $f \colon X \multimap Y$ is \textit{continuous} if it is both upper and lower semicontinuous. A continuous $n$-valued function is called an \textit{$n$-valued map.} For $n$-valued functions, continuity will be guaranteed even if the function is assumed only to be lower semi-continuous
 (cf. \cite{Brown2}).

\begin{definition}\label{def:split}
An $n$-valued map $f \colon X \multimap Y$ is \textit{split} if there are $n$ single valued continuous functions $f_1,\dots,f_n \colon X\to Y$ with $f(x) = \{ f_1(x),\dots,f_n(x) \}$ for every $x\in X$. In this case we write $f= \{f_1,\dots, f_n\}$. An $n$-valued map $f$ is \textit{locally split} if, given $x_0 \in X$ there is a neighborhood $U$ of $x_0$ such that $f \vert_U \colon U \multimap Y$ is split.
\end{definition}

\begin{lemma} \cite{Schirmer1}
If $X,Y$ are finite polyhedra, then any $n$-valued map $f \colon X \multimap Y$ is locally split. If additionally $X$ is simply connected, then any $n$-valued map $f \colon X \multimap Y$ is split.
\end{lemma}

An \textit{$n$-valued homotopy} between $n$-valued maps $f,f' \colon X \multimap Y$ is an $n$-valued map $H \colon X \times [0,1] \multimap Y$ satisfying $H(x,0) = f(x)$ and $H(x,1) = f'(x)$ for every $x \in X$. In this case we say that $f$ and $f'$ are \textit{homotopic}, and we write $f\simeq f'$. 

If $f \colon X \multimap Y$ is an $n$-valued map and $g \colon X \multimap Y$ is an $m$-valued map, then we call $(f,g)$ an \textit{$(n,m)$-valued pair of maps}. 

\begin{definition}\label{def:hom}
Two $(n,m)$-valued pairs of maps $(f,g)$ and $(f',g')$ are homotopic if there is an $n$-valued homotopy $\Phi \colon X \times [0,1] \multimap Y$ and an $m$-valued homotopy $\Psi \colon X \times [0,1] \multimap Y$ such that $\Phi_0 = f, \Phi_1 = f', \Psi_0 = g$ and $\Psi_1 = g'$. The pair $(\Phi,\Psi)$ is called an $(n,m)$-valued pair of homotopies.
\end{definition}

\subsection{The degree of $n$-valued self-maps of $S^1$}
In this paper we will concentrate mainly on $n$-valued self-maps of $S^1$. Below we will describe the idea of the \textit{degree} of $n$-valued maps of the circle defined in \cite{Brown3}.

Let $r \colon \mathbb{R} \to S^1 $ be given by $r(t) = e^{2\pi t i}$. We will denote points of the circle by $r(t)$ for $0 \leqslant t < 1$. Let $f \colon S^1 \multimap S^1$ be an $n$-valued map, then the $n$-valued function $f \circ r \colon [0,1] \multimap S^1$ is continuous and split and we write $f \circ r = \{f_0, f_1, . . . , f_{n-1}\}$ where the maps $f_j \colon [0,1] \to S^1$ have the property: $f_j (0) = r(t_j)$ for $0 \leqslant t_0 < t_1 < \ldots < t_{n-1} < 1$. Let $\widetilde{f_j} \colon [0,1] \to \R$ be the lift of $f_j$ such that $\widetilde{f_j}(0) = t_j$. We note that if $0 \leqslant j < k \leqslant n - 1$, then $\widetilde{f_j}(t) < \widetilde{f_k}(t)$ for all $t \in [0,1]$ because
$f_j(t) \neq f_k(t)$. Since $f$ is well-defined, the sets $(f \circ r)(0)$ and $(f \circ r)(1)$ are the same. Consequently, $\widetilde{f_0}(1) = v + t_J$ for some $v, J \in \mathbb{Z}$, where $0  \leqslant J \leqslant n - 1$. The degree of the $n$-valued map $f \colon S^1 \multimap S^1$ is defined as 
\begin{equation}\label{deg}
    \deg(f) = nv + J.
\end{equation}

The degree of an $n$-valued map defined by the formula (\ref{deg}) is a homotopy invariant.

\begin{remark}
    Let us notice, the degree of $f$ is equal to 
    \begin{equation} \label{eq:sumdegree}
        \sum\limits_{j = 0}^{n-1} (\widetilde{f_j}(1) - \widetilde{f_j}(0)).
    \end{equation} The sets $\widetilde{f_j}(1)$ and $\widetilde{f_j}(0)$ are equal modulo 1, so \eqref{eq:sumdegree} is an integer. Under a lifted homotopy, expression~\eqref{eq:sumdegree} is a continuous integer-valued function of the homotopy parameter, and hence is constant.
\end{remark}

\begin{theorem}\cite{Brown3}
    If $n$-valued maps $f, g$ are homotopic, then $\deg(f) = \deg(g)$.

\end{theorem}
Now we introduce the crucial class of so-called power maps, which will serve as a reference point for the study of general $n$-valued self-maps of the circle.

\begin{definition}\label{power}

Given integers $d$ and $n \geqslant 1$, the \textit{$n$-valued power map of degree $d$}, denoted $\phi_{n,d} \colon S^1 \multimap S^1$, is the $n$-valued map defined as follows:
\begin{equation}
\phi_{n,d}(r(t)) = \left\{ r\left(\frac{d}{n} t\right), r\left(\frac{d}{n} t + \frac{1}{n}\right), \ldots, r\left(\frac{d}{n} t + \frac{n-1}{n}\right)\right\},
\end{equation}
where $r \colon \mathbb{R} \to S^1$ is the covering map given by $r(t) = e^{i2\pi t}$.
\end{definition}

\noindent If we view the circle $S^1$ as the quotient group $S^1 = \R/\mathbb{Z}$, then every point from $S^1$ is represented by some $s \in \R$, and two real numbers $s_1, s_2$ represent the same point of the circle if and only if $s_1 - s_2 \in \mathbb{Z}$. For given $n \in \mathbb{N_+}$ and $d \in \mathbb{Z}$ using this representation we describe $\phi_{n,d}$ alternatively as follows:

\begin{equation}
    \phi_{n,d}(s) = \Big\{\frac{d}{n}s + \frac{u}{n};~~u \in \{0,1,\ldots,n-1\}  \Big\} \pmod 1,   
\end{equation}
\vspace{1.5pt}
for $s \in [0,1)$.

 The next lemma provides the value of the degree for an $n$-valued power map. The proof  follows directly from the definition of the degree.
\begin{lemma} \cite{Brown3}
 $\deg(\phi_{n,d}) =d$.
\end{lemma}

 One of the most important results in the theory of  $n$-valued power maps is the following  theorem, which provides a complete homotopy classification of $n$-valued maps of the circle, showing that each such map is determined, up to homotopy, by its degree.

\begin{theorem}[Classification Theorem,  \cite{Brown3}]
\label{th:classification}
    If $f \colon S^1 \multimap S^1$ is an $n$-valued map of degree $d$, then $f \simeq \phi_{n,d}$.
\end{theorem}

\section{The number of coincidences of an $(n,m)$-valued pair of power maps}\label{sec:power}

We begin by distinguishing two types of coincidence points. A point $x \in X$ is a \textit{fixed point} of an $n$-valued function $f \colon X \multimap X$ if $x \in f(x)$. In \cite{Brown1}, a point $x \in X$ is called a \textit{coincidence point} of an $n$-valued map $f \colon X \multimap Y$ and an $m$-valued map $g \colon X \multimap Y$ if $f(x) \cap g(x) \neq \emptyset$. To emphasize that this point $x$ is an element of the domain set $X$, we will call such a point a \textit{domain coincidence point}, and the \textit{domain coincidence set} of $f$ and $g$ is denoted by
\begin{equation}
    \label{coinset}
    \Coin_X(f:g) = \{ x \in X;~~f(x) \cap g(x) \neq \emptyset \} \subseteq X.
\end{equation}
We define the \textit{minimum number of domain coincidences} $\MC_X(f:g)$ as the minimum cardinality of $\Coin_X(f' : g')$ among all $(n,m)$-valued pairs of maps $(f',g')$ homotopic to $(f,g)$. 

For $(x,y) \in X \times Y$, we say $(x,y)$ is a \textit{graph intersection point} of $(f, g)$ if $y \in f(x) \cap g(x)$. The graph intersection set is defined as follows:
    \begin{equation}\label{graphcoin}
        \Coin_{X \times Y}(f:g) = \{ (x,y) \in X \times Y;~~y \in f(x) \cap g(x)\} \subseteq X\times Y.
    \end{equation}
The \textit{minimum number of graph intersections} $\MC_{X \times Y}(f:g)$ is defined to be the minimum cardinality of $\Coin_{X \times Y}(f' : g')$ among all $(n,m)$-valued pairs of maps $(f',g')$ homotopic to $(f,g)$.

    \begin{proposition}\label{MCXleMCXY}
    For an $n$-valued map $f \colon X \multimap Y$ and an $m$-valued map $g \colon X \multimap Y$, we have
    \begin{equation}
        \# \Coin_X(f:g) \leqslant \# \Coin_{X \times Y}(f:g).
    \end{equation}
\end{proposition}

\begin{proof}
There is at least one point $(x,y) \in \Coin_{X \times Y}(f:g)$ for every $x \in \Coin_X(f:g)$.
\end{proof}
Note that the example in Figure~\ref{introfig} demonstrates that $\#\Coin_X(f:g)$ need not equal $\#\Coin_{X\times Y}(f:g)$, which motivates the main theme of this paper.

In the forthcoming part of this section we determine $\Coin_X(\phi_{n,a} \colon \phi_{m,b})$. We obtain the same result as in \cite[Proposition 5.3]{Brown1} but our proof is different. Moreover our approach also enables us to compute $\Coin_{X\times Y}(\phi_{n,a} : \phi_{m,b})$.
 We start by proving two lemmas.

\begin{lemma}
\label{lemma:main}
Let $n,m \geq 1$ and $\GCD(n,m)=1$ and $g\in \mathbb{Z}$. Then there exists $c,d\in \mathbb N$ with $0\le c < n$ and $0 \le d < m$ such that
\begin{equation}
    g \equiv (nd - mc) \pmod{nm}.
    \label{eq:mainlemma}
\end{equation}
\end{lemma}


\begin{proof}
    Since $\GCD(n,m)=1$, from Bézout's identity there are integers $a,b$ with $an + bm = 1$. Then we have $gan + gbm = g$. Let $d$ be the remainder of $ga$ after division by $m$ and $c$ be the remainder of $-gb$ after division by $n$, so $ga = km + d$ and $-gb = ln + c$ for some integers $k,l$ and $0\le c < n$, $0 \le d < m$. We have
    \begin{equation*}
    \begin{aligned}
    g &= gan + gbm= kmn + dn - lnm - cm\\
    &= (nd - mc) + (k-l)nm\equiv (nd - mc) \pmod{nm}.        
    \end{aligned}
    \end{equation*}
\end{proof}

\begin{lemma}
\label{lemma:copies}
    Let $\phi_{n,a} \colon S^1 \multimap S^1$ be an $n$-valued power map of degree $a$ and $w \in \mathbb{N}$. If $w \mid n$, then for every $t \in \{0,1,\ldots,w-1\}$ and every $s \in [0,1)$, we have the following equality of sets:
    \begin{equation}
    \label{eq:copies}
        (\phi_{n,a}(s) + \frac{t}{w}) \pmod 1 = \phi_{n,a}(s),
    \end{equation}
    where the symbol $+$ means addition of the number $\frac{t}{w}$ to every element of the set $ \phi_{n,a}(s)$.
\end{lemma}

\noindent The lemma above means that the graph of $\phi_{n,a}$ splits vertically into $w$ different identical blocks. The following figure illustrates this fact for the example of $\phi_{6,2}$ and $w = 3$. This is a straightforward fact, so we omit the proof.

In order to compute the number of coincidences for an $(n,m)$-valued pair of maps, we first consider the case $am = bn$.

\begin{lemma}
\label{lemma:domainS1}
Let $\phi_{n,a},\phi_{m,b}\colon S^1\multimap S^1$ be power maps. If \(am=bn\), then
\begin{equation}
\Coin_X(\phi_{n,a}\colon\phi_{m,b})=S^1. 
\end{equation}
Moreover, there exists \(\varepsilon>0\) such that
\begin{equation}
\Coin_X(\phi_{n,a}\colon\phi_{m,b}+\varepsilon)=\emptyset.  
\end{equation}
\end{lemma}

\begin{figure}[ht]
    \centering
    \includegraphics[width = 0.75\textwidth]{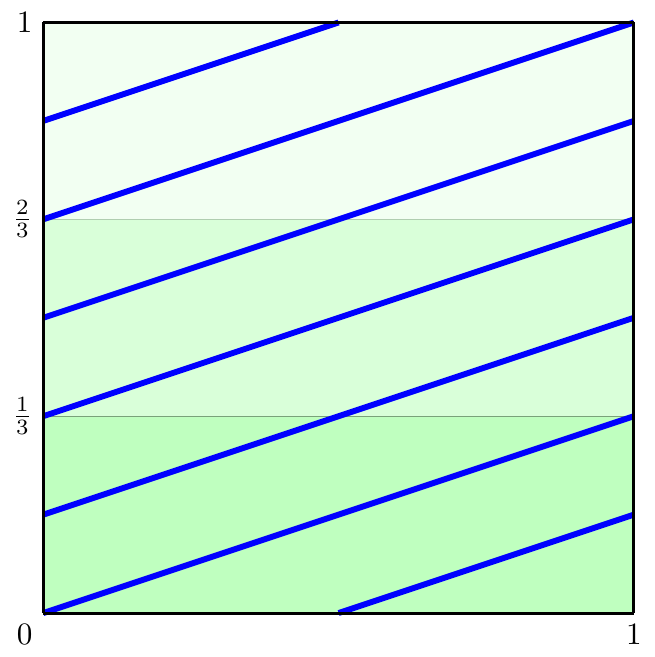}
    \caption{Three identical vertical blocks of $\phi_{6,2}$.}
\end{figure}

\begin{proof}
    Since $am = bn$, all components of $\phi_{n,a}$ and $\phi_{m,b}$ are affine maps with the same slope ($\frac{a}{n}$ and $\frac{b}{m}$, respectively), separated vertically by $\frac{1}{n}$ and $\frac{1}{m}$, respectively. The components which contain the point $(0,0)$ coincide, and therefore
    \begin{equation*}
\Coin_X(\phi_{n,a}\colon\phi_{m,b})=S^1.
    \end{equation*}

If, for instance, $\varepsilon = \frac{1}{2nm}$, then no component of $\phi_{n,a}$ intersects any component of $\phi_{m,b} + \varepsilon$, therefore
    \begin{equation*}
\Coin_X(\phi_{n,a}\colon\phi_{m,b}+\varepsilon)=\emptyset.
    \end{equation*}
\end{proof}

Now we consider the case of $am \neq bn$ and $\GCD(n,m) = 1$. 

\begin{lemma}\label{lemma:Coin_X}
      If $am \neq bn$ and $\GCD(n,m) = 1$, then
    \begin{equation}
        \Coin_X(\phi_{n,a} : \phi_{m,b}) = \left\{0, \frac{1}{k}, \dots, \frac{k-1}{k}   \right\},
    \end{equation}
    where $k = \vert am - bn \vert$. Thus we have  
        \begin{equation}
        \#\Coin_X(\phi_{n,a} : \phi_{m,b}) = \vert am - bn \vert.
    \end{equation}
\end{lemma}

\begin{proof}
    We are looking for $s \in [0,1)$ for which $\phi_{n,a}(s) \cap \phi_{m,b}(s) \neq \emptyset$. So,

\begin{subequations} \label{eqn:maincong}
    \begin{align*}
     s \in \Coin_X(\phi_{n,a} : \phi_{m,b}) \iff \exists_{u, v}~~  \Big(\frac{a}{n}s + \frac{u}{n}\Big) &\equiv \Big( \frac{b}{m}s + \frac{v}{m} \Big) \pmod 1,  \\
      \frac{as + u}{n} &\equiv \frac{bs+v}{m} \pmod 1, \\
      (am - bn)s &\equiv (v n - u m) \pmod{nm}. \tag{\ref{eqn:maincong}} 
      \hspace{1cm} 
    \end{align*}
\end{subequations}

\noindent The right-hand side of the formula above in the  line (\ref{eqn:maincong}) is an integer, so the left-hand side also must be an integer. This means that there are exactly $|am-bn|$ solutions of $s$ in $[0,1)$. These solutions must have the form  $s = \frac{0}{k}, \frac{1}{k}, \ldots, \frac{k-1}{k}$, where $k = |a m - b n|$. By Lemma~\ref{lemma:main} for every solution $s$ there exist $u, v$ satisfying the formula~\eqref{eqn:maincong}. As a consequence,  in this case every solution $s = \frac{0}{k}, \frac{1}{k}, \ldots, \frac{k-1}{k}$, where $k = |a m - b n|$, must be a different domain coincidence point.
\end{proof}


\noindent It is easy to see that every solution $s$ from the proof of Lemma~\ref{lemma:Coin_X} is realized by a unique pair $(u, v)$. As a consequence every $s$ generates exactly one $(x,y) \in \Coin_{X \times Y}(\phi_{n,a} : \phi_{m,b})$. So we conclude the following:

\begin{corollary} \label{cor:Coin_XY}
          If $am \neq bn$ and $\GCD(n,m) = 1$, then
    \begin{equation}
        \# \Coin_{X \times Y}(\phi_{n,a} : \phi_{m,b}) = \vert am - bn \vert.
    \end{equation}       
\end{corollary}

Now we consider the case of $\GCD(n,m) > 1$. Given two power maps $\phi_{n,a},\phi_{m,b}$ with $am \neq bn$, let $w= \GCD(n,m) \neq 1$. Let $\bar\phi_{n,a}(s) = \phi_{n,a}(s) \cap [0, \frac{1}{w})$ and $\bar\phi_{m,b}(s) = \phi_{m,b}(s) \cap [0, \frac{1}{w})$, so we consider the bottom blocks of the graph of given power maps, as described in Lemma \ref{lemma:copies}. 

\begin{remark}
    Let us notice $\phi_{n,a}$ is composed of $w$ vertical translations of $\bar\phi_{n,a}$.
\end{remark}

\begin{lemma} \label{lemma:CoinXbottom}
            If $am \neq bn$ and $\GCD(n,m) = w \neq 1$, then
    \begin{equation}
        \Coin_X(\bar\phi_{n,a} : \bar\phi_{m,b}) = \Big\{0, \frac{w}{k}, \frac{2w}{k}, \ldots, \frac{k-w}{k}  \Big\},
    \end{equation}
    where $k = \vert am - bn \vert$. Thus we have  
        \begin{equation}
        \# \Coin_X(\bar\phi_{n,a} : \bar\phi_{m,b}) = \frac{\vert am - bn \vert}{w}.
    \end{equation} 
\end{lemma}

\begin{proof}
Write
\[
n=wn_0,\qquad m=wm_0,
\]
where $\GCD(n_0,m_0)=1$, and put
\[
D_0=am_0-bn_0=\frac{am-bn}{w}.
\]
The coincidence condition is equivalent to
\[
D_0s\equiv vn_0-um_0\pmod{wn_0m_0}.
\]
Hence $D_0s$ must be an integer. Conversely, since
$\GCD(n_0,m_0)=1$, Bézout's identity shows that every integer can be
written in the form $vn_0-um_0$. Reducing $u$ modulo $n$ and $v$
modulo $m$ does not change this expression modulo $wn_0m_0$.
Therefore the possible values of $s$ are exactly
\[
s=\frac{j}{|D_0|},
\qquad
j=0,1,\ldots,|D_0|-1.
\]
For each such $s$, replacing $(u,v)$ by
\[
(u+tn_0,v+tm_0),\qquad t=0,\ldots,w-1,
\]
with the indices taken modulo $n$ and $m$, does not change the
coincidence condition and shifts the common value by $\frac tw$.
Hence exactly one of these common values lies in the bottom block
$[0,\frac1w)$. Consequently,
\[
\Coin_X(\bar\phi_{n,a}:\bar\phi_{m,b})
=
\left\{
0,\frac{w}{k},\frac{2w}{k},\ldots,\frac{k-w}{k}
\right\},
\qquad k=|am-bn|,
\]
and
\[
\#\Coin_X(\bar\phi_{n,a}:\bar\phi_{m,b})
=
|D_0|
=
\frac{|am-bn|}{w}.
\]
\end{proof}

\noindent 
Since every domain coincidence $s \in \Coin_X(\bar\phi_{n,a} : \bar\phi_{m,b})$ corresponds to a unique graph intersection, the following corollary is immediate.

\begin{corollary}
           If $am \neq bn$ and $\GCD(n,m) = w \neq 1$, then
    \begin{equation}
        \# \Coin_{X \times Y}(\bar\phi_{n,a} : \bar\phi_{m,b}) = \frac{\vert am - bn \vert}{w}.
    \end{equation}         
\end{corollary}

Since $\Coin_X$ is invariant under simultaneous translation, we have
\[
\Coin_X\!\left(\bar\phi_{n,a}+\frac{t}{w}:\bar\phi_{m,b}+\frac{t}{w}\right)
=
\Coin_X(\bar\phi_{n,a}:\bar\phi_{m,b}),
\]
and thus we obtain:

\begin{corollary} \label{cor:Coin_X}
             If $am \neq bn$ and $\GCD(n,m) \neq 1$, then
        \begin{equation}
        \# \Coin_X(\phi_{n,a} : \phi_{m,b}) = \frac{\vert am - bn \vert}{\GCD(n,m)}.
    \end{equation}  
\end{corollary}


Now we consider the set $\Coin_{X \times Y}(\phi_{n,a} : \phi_{m,b})$.

\begin{lemma} \label{lemma:Coin_XY}
If $am \neq bn$ and $\GCD(n,m) = w \neq 1$, then
    \begin{equation}
        \# \Coin_{X \times Y}(\phi_{n,a} : \phi_{m,b}) = \frac{\vert am - bn \vert}{w} \cdot w = \vert am - bn \vert.
    \end{equation}        
\end{lemma}

\begin{proof}
We start with a simple observation. 

\begin{equation*}
    \Coin_{X \times Y}\Big(\bar\phi_{n,a} + \frac{j}{w} \colon \bar\phi_{m,b}+ \frac{l}{w}\Big) = \emptyset,
\end{equation*}
for $j \neq l$, $j,l \in \{0,1,\ldots,w-1\}$. This is obvious because the second coordinates of both maps will never be equal. We have $\Coin_{X \times Y}\Big(\bar\phi_{n,a} + \frac{j}{w} \colon \bar\phi_{m,b} + \frac{l}{w}\Big) \neq \emptyset$ if and only if $j = l$. So for every $t \in \{0,1,\ldots,w-1\}$, we have
\begin{equation*}
    \# \Coin_{X \times Y}\Big(\bar\phi_{n,a} + \frac{t}{w} \colon \bar\phi_{m,b} + \frac{t}{w}\Big) = \frac{\vert am - bn \vert}{w}.
\end{equation*}
As there are $w$ different choices for $t$, we obtain the final result.
\end{proof}

Combining the results for $\GCD(n,m) = 1$ and $\GCD(n,m) \neq 1$, we obtain the following theorems.

\begin{theorem}\label{th:numberofcoinx}
         If $am \neq bn$, then
    \begin{equation}
        \Coin_X(\phi_{n,a} : \phi_{m,b}) = \Big\{0, \frac{w}{k}, \frac{2w}{k}, \ldots, \frac{k-w}{k}  \Big\},
    \end{equation}
    where $k = \vert am - bn \vert$, $w = \GCD(n,m)$. Thus
        \begin{equation}
        \# \Coin_X(\phi_{n,a} : \phi_{m,b}) = \frac{\vert am - bn \vert}{\GCD(n,m)}.
    \end{equation}  
\end{theorem}

\begin{proof}
    It follows from Lemma~\ref{lemma:Coin_X} and Corollary~\ref{cor:Coin_X}
\end{proof}


\begin{theorem}
\label{th:numberofcoinxy}
         If $am \neq bn$, then
    \begin{equation}
        \# \Coin_{X \times Y}(\phi_{n,a} : \phi_{m,b}) = \vert am - bn \vert.
    \end{equation}  
\end{theorem}

\begin{proof}
    It follows from Corollary~\ref{cor:Coin_XY} and Lemma~\ref{lemma:Coin_XY}.
\end{proof}


\section{The geometric Nielsen coincidence number}\label{sec:minimum}

We now define the geometric Nielsen coincidence number, which gives a natural homotopy-invariant lower bound for the minimal number of graph intersection points and serves as the geometric starting point for our later considerations. The notion of a graph intersection class introduced here provides a corrected version of the coincidence class defined in \cite{Brown1}.

Assume that $f,g \colon X \multimap Y$ are an $(n,m)$-valued pair of maps of connected finite polyhedra. We define a relation called \textit{graph Nielsen relation} in the set $\Coin_{X \times Y}(f:g)$ as follows: the points $(x,y), (x',y')$ are equivalent graph intersection points if there is a path $p \colon [0,1] \to X$ such that $p(0) = x,~~p(1) = x'$ and for the splittings $f \circ p = \{f_1,\ldots,f_n\}$ and $g \circ p  = \{g_1,\ldots,g_m\}$, there exist $1 \leqslant j \leqslant n$ and $1 \leqslant k \leqslant m$ such that $f_j(0) = g_k(0) = y$, $f_j(1) = g_k(1) = y'$ and the paths $f_j, g_k \colon [0,1] \to Y$ are homotopic relative to the endpoints.

\begin{lemma}
The  graph Nielsen relation  is an equivalence relation.
\end{lemma}

\begin{proof}
    The reflexivity and symmetry are obvious. We will show the transitivity. Assume that points $(x_1,y_1)$, $(x_2,y_2)$ and $(x_2,y_2)$, $(x_3,y_3)$ are in the relation. Let $p_1 \colon [0,1] \to X$ be the path from $x_1$ to $x_2$ which satisfies the definition of the relation, and $p_2 \colon [0,1] \to X$ be the analogous path from $x_2$ to $x_3$. We  define the path $p_3 \colon [0,1] \to X$ in the standard way as $p_3=p_1 * p_2$, where $*$ represents path concatenation. 

Since $(x_1,y_1)$, $(x_2,y_2)$ and $(x_2,y_2)$, $(x_3,y_3)$ are related, we obtain maps $f^1_{i} \colon [0,1] \to Y$ from the splitting of $f \circ p_1$, and $f^2_{j}\colon [0,1] \to Y$ from the splitting of $f \circ p_2$. These maps will satisfy $f^1_{i}(1) = y_2 = f^2_{j} (0)$. 

Now let $f^3_k =  f^1_i * f^2_j$, and observe that this map appears in a splitting of $f\circ p_3$.  Analogously there exists $g^3_{{\bar k}} \colon [0,1] \to Y$ given by $g^3_{{\bar k}} = g^1_{\bar i} * g^2_{\bar j}$. The paths $f^3_{k},g^3_{{\bar k}}$ are homotopic relative to the endpoints, and so $(x_1,y_1), (x_3,y_3)$ are in the relation. Thus the transitivity is satisfied.
\end{proof} 

The corresponding equivalence classes are called the \textit{graph intersection classes} of $f$ and $g$.  By the standard compactness argument in Nielsen theory, using the local
splitting property and the compactness of \(X\times Y\), the number of graph
intersection classes is finite.

The next result corresponds to Lemma 6.2 from \cite{Schirmer1} in respect to coincidences of $(n,m)$-valued pairs of maps. It is a modification (using graph intersections instead of domain coincidences) of Lemma 2.1 from \cite{Brown1}.

\begin{lemma} \label{lemma:classeshomotopy}
    Let $\Phi, \Psi \colon X \times [0,1] \multimap Y$ be an $(n,m)$-valued pair of homotopies. The intersection $C_t$ of a graph intersection class $C$ of $\Phi$ and $\Psi$ with $X \times \{t\} \times Y$, where $t \in [0,1]$, is either empty or a graph intersection class of $\Phi_t$ and $\Psi_t$. Each graph intersection class of $\Phi_t$ and $\Psi_t$ is contained in a unique graph intersection class of $\Phi$ and $\Psi$.
\end{lemma}

\begin{proof}
    It is sufficient 
    to show that two points $((x_1,t_0),y_1), ((x_2,t_0),y_2)$ are in the same graph intersection class of $\Phi$ and $\Psi$ if and only if the points $(x_1,y_1), (x_2,y_2)$ are in the same graph intersection class of $\Phi_{t_0}$ and $\Psi_{t_0}$.

    Suppose that $((x_1,t_0),y_1), ((x_2,t_0),y_2)$ are in a graph intersection class $C$ of $\Phi$ and $\Psi$, so there is a path $p \colon [0,1] \to X \times [0,1]$ such that $p(0) = (x_1,t_0),~~p(1) = (x_2,t_0)$, and for the splittings $\Phi \circ p = \{\phi_1,\ldots,\phi_n\}$ and $\Psi \circ p = \{\psi_1,\ldots,\psi_m\}$, there exist $1 \leqslant j \leqslant n$ and $1 \leqslant k \leqslant m$ such that $\phi_j(0) = \psi_k(0) = y_1$, $\phi_j(1) = \psi_k(1) = y_2$, and the paths $\phi_j, \psi_k \colon [0,1] \to Y$ are homotopic relative to the endpoints. We order the maps from splittings so that $k = 1$ and $j = 1$, and we denote the homotopy between $\phi_1$ and $\psi_1$ as $M \colon [0,1] \times [0,1] \to Y$. 
    
    Let $\pi \colon X \times [0,1] \to X \times \{ t_0 \}$ be the projection map. Let us define $\widetilde{p} = \pi \circ p \colon [0,1] \to X \times \{ t_0 \} \subset X \times [0,1]$. Clearly, the paths $p$ and $\widetilde p$ are homotopic relative to the
endpoints, for instance by a linear homotopy in the second coordinate.
Let
\[
H\colon [0,1]\times[0,1]\longrightarrow X\times[0,1]
\]
be a homotopy from $\widetilde p$ to $p$ relative to the endpoints; thus
\[
H(t,0)=\widetilde p(t),
\qquad
H(t,1)=p(t).
\]

Since $[0,1]\times[0,1]$ is simply connected, the multivalued maps
$\Phi\circ H$ and $\Psi\circ H$ split. Choose the first branches of these
splittings, denoted by $(\Phi\circ H)_1$ and $(\Psi\circ H)_1$,
respectively, so that
\[
(\Phi\circ H)_1(x,0)=\widetilde\phi_1(x),
\qquad
(\Phi\circ H)_1(x,1)=\phi_1(x),
\]
and
\[
(\Psi\circ H)_1(x,0)=\widetilde\psi_1(x),
\qquad
(\Psi\circ H)_1(x,1)=\psi_1(x).
\]

    The desired homotopy $L \colon [0,1] \times [0,1] \to Y$ between $\widetilde{\phi}_1$ and $\widetilde{\psi}_1$ is given by the following combination
    \begin{equation*}
        L(x,t) = \left\{ \begin{array}{lr} (\Phi \circ H)_1(x,3t) & \textrm{for} \ 0 \leq t < \frac{1}{3}, \\ M (x,3t-1) & \textrm{for} \ \frac{1}{3} \leq t < \frac{2}{3}, \\ (\Psi \circ H)_1(x,-3t+3) & \textrm{for} \ \frac{2}{3} \leq t \leq 1. \end{array}\right.
    \end{equation*}

    It is easy to see that $L(0,t) = \widetilde{\phi}_1(0) = \widetilde{\psi}_1(0)$ and $L(1,t) = \widetilde{\phi}_1(1) = \widetilde{\psi}_1(1)$, so $\widetilde{\phi}_1$ and $\widetilde{\psi}_1$ are homotopic relative to the endpoints (as all homotopies inside the formula of $L$ are relative to the endpoints).
For any given $t_0$ there exists a path $\widetilde{p} \colon [0,1] \to X \times \{ t_0 \}$ for which $\widetilde{p}(0) = (x_1,t_0)$, $\widetilde{p}(1) = (x_2,t_0)$ and for the splittings $\Phi \circ \widetilde{p} = \{\widetilde{\phi}_1,\ldots,\widetilde{\phi}_n\}$ and $\Psi \circ \widetilde{p} = \{\widetilde{\psi}_1,\ldots,\widetilde{\psi}_m\}$, there exist $\widetilde{\phi}_1(0) = \widetilde{\psi}_1(0) = y_1$, $\widetilde{\phi}_1(1) = \widetilde{\psi}_1(1) = y_2$ and the paths $\widetilde{\phi}_1, \widetilde{\psi}_1 \colon [0,1] \to Y$ are homotopic relative to the endpoints, so the points $(x_1,y_1)$ and $(x_2,y_2)$ are in the same graph intersection class $C_{t_0}$ of $\Phi_{t_0}$ and $\Psi_{t_0}$. We have shown that the nonempty intersection $C_{t_0} = C \cap (X \times \{ t_0 \} \times Y)$ is a graph intersection class of $\Phi_{t_0}$ and $\Psi_{t_0}$.

Conversely, suppose that points $(x_1,y_1)$ and $(x_2,y_2)$ are in the same graph intersection class of $\Phi_{t_0}$ and $\Psi_{t_0}$. Then there exists a path $p \colon [0,1] \to X \times \{ t_0 \}$, which satisfies all the conditions from the definition of  equivalence of graph intersection classes. It is easy to see this path $p \colon [0,1] \to X \times \{ t_0 \} \subset X \times [0,1]$ also demonstrates the equivalence of the graph intersection classes of $\Phi$ and $\Psi$. The splittings of $\Phi \circ p$, $\Psi \circ p$ are the same as the splittings of $\Phi_{t_0} \circ p$ and $\Psi_{t_0} \circ p$, so the points $((x_1,t_0),y_1)$ and $((x_2,t_0),y_2)$ are also in the same graph coincidence class of $\Phi$ and $\Psi$, thus, they are in a coincidence class $C_{t_0} = C \cap (X \times \{ t_0\} \times Y)$. 
\end{proof}

\subsection{The geometric Nielsen number}

A graph intersection class $C_0$ of an $(n,m)$-valued pair of maps $(f,g) \colon X \multimap Y$ is \textit{geometrically inessential} if there are homotopies $\Phi, \Psi \colon X \times [0,1] \multimap Y$ such that $\Phi_0 = f, \Psi_0 = g$ and the graph intersection class $C$ of $\Phi$ and $\Psi$ containing $C_0$ has the property $C \cap (X \times \{1\} \times Y) =  \emptyset$. Otherwise, the graph intersection class is \textit{geometrically essential}. The \textit{geometric Nielsen coincidence number} $\hat N(f:g)$ is the number of geometrically essential graph intersection classes.

In \cite{Brown1} the Nielsen relation is defined on $\Coin_X(f:g)$, but it cannot be properly defined on this set. We demonstrated this briefly in the introduction (Figure~\ref{introfig}) and now we revisit the idea with the following example. 

\begin{example}
Consider the pair of power maps $(\phi_{2,1}, \phi_{3,-1})$ and the pair $(f,g)$, where the graphs of the maps $f,g \colon S^1 \multimap S^1$ are presented on the figure, along with the domain coincidences of both of the pairs. 

\begin{figure}[ht]
    \centering
    \includegraphics[width=0.75\linewidth]{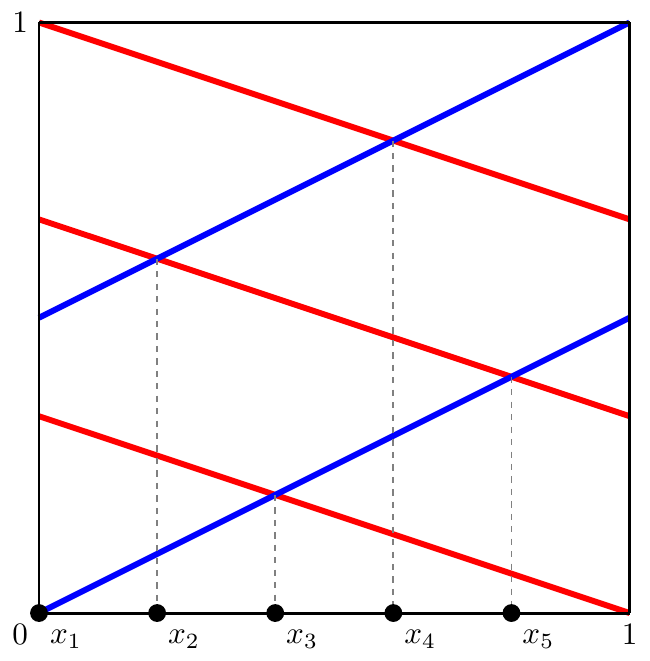}
  \caption{Domain coincidences of $(\phi_{2,1}, \phi_{3,-1})$.\label{fig:left}}

\end{figure}

\begin{figure}[ht]
    \centering
    \includegraphics[width=0.75\linewidth]{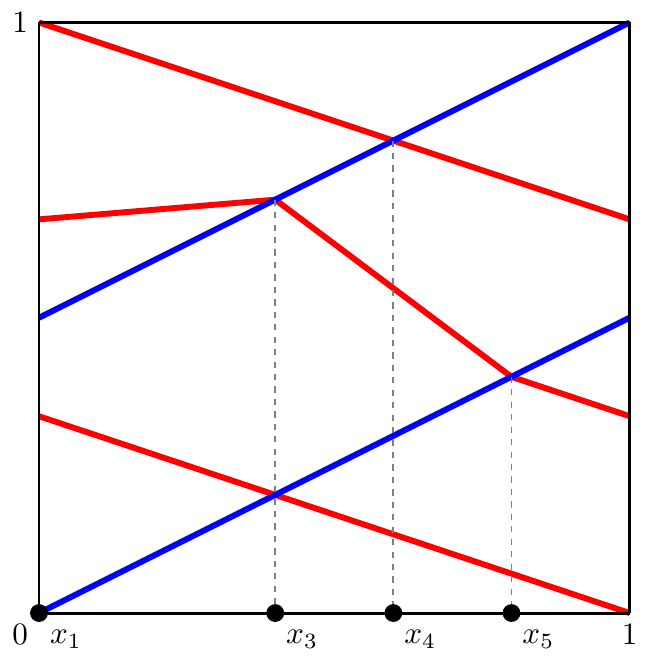}
  \caption{Domain coincidences of $(f, g)$.\label{fig:counter2}}
\end{figure}

\noindent As $\deg(f) = \deg(\phi_{2,1}) = 1$ and $\deg(g) = \deg(\phi_{3,-1})= -1$, these pairs of maps are homotopic, so the Nielsen numbers are equal. Using the formula \eqref{browneq} which was claimed by Brown \& Kolahi in \cite{Brown1}, this Nielsen number would be calculated as
\begin{equation}\label{eq:N}
\frac{\vert 1 \cdot 3 - (-1) \cdot 2 \vert}{\GCD(2,3)} = 5.
\end{equation}

\noindent However, Brown \& Kolahi's Nielsen number is intended to be a lower bound on the number of domain coincidence points, but we can clearly see that $\#\Coin_X(f:g)=4 < 5$. Indeed checking the details of the Nielsen relation on $\Coin_X(\phi_{2,1} : \phi_{3,-1})$ from \cite{Brown1} will show that none of these domain coincidence points are equivalent, but, as our example shows, two of them can be merged into one point by a homotopy.
\end{example}

\begin{remark}
The core difficulty with the Nielsen relation discussed in \cite{Brown1} lies in its non-transitivity, as the relation is defined solely in terms of points in the domain. In our example: if we consider the pair of linear homotopies $\Phi, \Psi \colon S^1 \times [0,1] \multimap S^1$ such that $\Phi_0 = \phi_{2,1}, \Psi_0 = \phi_{3,-1}$ and $\Phi_1 = f, \Psi_1 = g$ (so Figure~\ref{fig:left} shows the homotopies in time $t=0$ and Figure~\ref{fig:counter2} shows the homotopies in time $t=1$), directly by the definition the (domain) coincidence points $(x_3,0), (x_3,1)$ are related, as well as the points $(x_2,0), (x_3,1)$, but the  relation does not hold between the points $(x_3,0), (x_2,0)$. 
\end{remark}

The following theorem is proved in \cite{Brown1}. We restate that proof using graph intersection classes and the geometric Nielsen coincidence number. Apart from this reformulation, the argument is essentially the same as that given in \cite{Brown1}.

\begin{theorem} (cf. \cite{Brown1})
    The geometric Nielsen coincidence number for $(n,m)$-valued pairs of maps of connected finite polyhedra is a homotopy invariant. If $f,f' \colon X \multimap Y$ and $g,g' \colon X \multimap Y$ are homotopic, respectively, then
    \begin{equation}
        \hat N(f:g) = \hat N(f' : g').
    \end{equation}
\end{theorem}

\begin{proof}
    Let $\Phi, \Psi \colon X \times [0,1] \multimap Y$ be homotopies such that $\Phi_0 = f, \Phi_1 = f', \Psi_0 = g$ and $\Psi_1 = g'$. Let us denote an essential coincidence class of $f$ and $g$ by $C_0$. From Lemma~\ref{lemma:classeshomotopy} there exists a coincidence class of $\Phi$ and $\Psi$ containing $C_0$. We denote it by $\mathbf{C}$. Then $C_1 =
 \mathbf{C} \cap (X \times \{1\} \times Y)$ is a coincidence class of $f'$ and $g'$. If $C_1$ was an inessential coincidence class of $f'$ and $g'$, then there would be an $(n,m)$-valued pair of
 homotopies $\Phi', \Psi'\colon X \times [0,1] \multimap Y$ such that for the coincidence class of $\Phi'$ and $\Psi'$, denoted by $\mathbf{C'}$ and containing $C_1$, we would have $ \mathbf{C'} \cap (X \times \{1\} \times Y) = \emptyset$. Now we define homotopies $\Phi'', \Psi'' \colon X \times [0,1] \multimap Y$ by
     \begin{equation*}
           \Phi''(x,t) = \left\{ \begin{array}{lcr}\Phi(x,2t) & \textrm{for} & \ t \in [0,\frac12] \\ \Phi'(x,2t-1) & \textrm{for} & \ t \in [\frac12,1], \end{array}\right.
    \end{equation*}
 and $\Psi''$ analogously.  Then there exists a graph intersection class $\mathbf{C''}$ of $(\Phi'',\Psi'')$
whose restriction to $t \in [0,\tfrac12]$ corresponds to $\mathbf{C}$
and whose restriction to $t \in [\tfrac12,1]$ corresponds to $\mathbf{C'}$.  However $ \mathbf{C''} \cap (X \times \{1\} \times Y) = \emptyset$. Therefore, $C_1$ is also essential and the pair of homotopies $(\Phi,\Psi)$ determine a one-to-one correspondence between the essential coincidence classes of $f$ and $g$ and the essential coincidence classes of $f'$ and $g'$. The $(n,m)$-valued pair of homotopies $\bar{\Phi}, \bar{\Psi} \colon X \times [0,1] \multimap Y$ defined by $\bar{\Phi}(x,t) = \Phi(x,1-t)$ and $\bar{\Psi}(x,t) = \Psi(x,1-t)$ determines a one-to-one correspondence between the essential coincidence classes of $f'$ and $g'$ and the essential coincidence classes of $f$ and $g$ in the same way. We conclude $\hat N(f:g) = \hat N(f'\colon g')$.
\end{proof}


\begin{proposition}\label{oszac}
        Since each $(n,m)$-valued pair $(f',g')$ homotopic to $(f,g)$ has at least $\hat N(f:g)$ essential graph intersection classes, and in every essential class there is at least one graph intersection point, we have
    \begin{equation}
        \hat N(f:g) \leq \MC_{X \times Y} (f:g) .
    \end{equation}
\end{proposition}

\noindent By the Classification Theorem \ref{th:classification} and Proposition \ref{oszac} we conclude that for an $n$-valued map $f \colon S^1 \multimap  S^1$ with degree $a$ and an $m$-valued map $g \colon  S^1 \multimap  S^1$ with degree $b$, we have
\begin{equation}
\label{eq:mcxybound}
    \hat  N (f:g) = \hat  N (\phi_{n,a} : \phi_{m,b}) \leq \MC_{X \times Y} (\phi_{n,a} : \phi_{m,b}).
\end{equation}

Now we will show that every graph intersection point of two power maps is in a different graph intersection class, which is one of the crucial observations.  

\begin{lemma} \label{lemma:difclass}
    If $am \neq bn$, then every $(x,y) \in \Coin_{X \times Y} (\phi_{n,a} : \phi_{m,b})$ is in a different graph intersection class of $\phi_{n,a}$ and $\phi_{m,b}$.   
\end{lemma}

\begin{proof}
Let us consider two different graph intersection points $(x,y)$ and $(x',y')$
of the pair $(\phi_{n,a}, \phi_{m,b})$. The points $(x,y)$ and $(x',y')$
are equivalent graph intersection points if there is a path
$p \colon [0,1] \to X$ such that $p(0) = x$, $p(1) = x'$ and for the
splittings $
\phi_{n,a} \circ p = \{f_1,\ldots,f_n\}$ and $
\phi_{m,b} \circ p  = \{g_1,\ldots,g_m\},$
there exist $1 \leqslant i \leqslant n$ and
$1 \leqslant j \leqslant m$ such that $f_i(0) = g_j(0) = y,
f_i(1) = g_j(1) = y'$
and the paths $f_i, g_j \colon [0,1] \to Y$ are homotopic relative to the
endpoints.

Assume contrary to our claim that there exist such $f_i$ and $g_j$.
If $am - bn \neq 0$, then
\[
\frac{a}{n} \neq \frac{b}{m},
\]
so $f_i, g_j$ are compositions of the path $p$ with linear functions with
different slopes. Let $\widetilde f_i, \widetilde g_j \colon [0,1] \to \mathbb{R}$ be lifts
of $f_i, g_j$ and $\widetilde p$ be a lift of $p$. Since $f_i$ and $g_j$ are branches
of the power maps $\phi_{n,a}$ and $\phi_{m,b}$ along the same path $p$,
their lifts may be written in the form
\[
\widetilde f_i(t)=\frac an\,\widetilde p(t)+c_f,
\qquad
\widetilde g_j(t)=\frac bm\,\widetilde p(t)+c_g,
\]
for suitable constants $c_f,c_g\in\mathbb R$. We choose the lifts so that
\[
\widetilde f_i(0)=\widetilde g_j(0).
\]
Therefore
\[
\begin{aligned}
&\bigl(\widetilde f_i-\widetilde g_j\bigr)(1)
-\bigl(\widetilde f_i-\widetilde g_j\bigr)(0)  =
\left(\frac an-\frac bm\right)
\bigl(\widetilde p(1)-\widetilde p(0)\bigr).
\end{aligned}
\]

Suppose now that
$\widetilde f_i(1)=\widetilde g_j(1)$. Since also
$\widetilde f_i(0)=\widetilde g_j(0)$, the left-hand side of the above
equality is zero. Hence
\[
0=
\left(\frac an-\frac bm\right)
\bigl(\widetilde p(1)-\widetilde p(0)\bigr).
\]
Since $\frac an\neq\frac bm$, we obtain $\widetilde p(1)=\widetilde p(0)$ and since $p(0)=x$ and $p(1)=x'$, projecting the last equality to $S^1$ gives
\[
x'=p(1)=p(0)=x.
\]
Moreover,
\[
\widetilde f_i(1)-\widetilde f_i(0)
=
\frac an\bigl(\widetilde p(1)-\widetilde p(0)\bigr)=0,
\]
and therefore
\[
y'=f_i(1)=f_i(0)=y.
\]
Thus $(x',y')=(x,y)$, contrary to the assumption that the two graph
intersection points are distinct. Consequently,
\[
\widetilde f_i(1)\neq\widetilde g_j(1).
\]

However, if $f_i,g_j$ are homotopic relative to the endpoints, then the
lifts $\widetilde f_i,\widetilde g_j$ have also this property. This leads
to a contradiction because
$\widetilde f_i(1)\neq\widetilde g_j(1)$.
\end{proof}

\begin{proposition}
\label{prop:N=0}
        Let $f \colon S^1 \multimap S^1$ be an $n$-valued map with degree $a$, and $g \colon S^1 \multimap S^1$ be an $m$-valued map with degree $b$ such that $am = bn$. Then
    \begin{equation}
        \hat N( f \colon g ) = 0.
    \end{equation}    
\end{proposition}


\begin{proof}
    Directly from Lemma~\ref{lemma:domainS1} we get a pair of maps $(\phi_{n,a},\varphi)$ homotopic to $(\phi_{n,a},\phi_{m,b})$ without domain coincidence points, and as a result, without graph intersection points. Using the formula~\eqref{eq:mcxybound}, we obtain
    \begin{equation*}
        \hat N (f:g) = \hat N (\phi_{n,a} : \phi_{m,b}) = \hat N (\phi_{n,a} : \varphi)  \leq \MC_{X \times Y} (\phi_{n,a} : \varphi) = 0. \qedhere
    \end{equation*}
    \end{proof}

In the case $am \neq bn$  determining geometric Nielsen coincidence number becomes significantly more challenging. The main difficulty lies in proving the essentiality of the geometric Nielsen coincidence classes. For this reason, in the following sections, we propose a different approach, concentrating on determining the minimal number of intersection points in the homotopy class (cf. \textit{algebraic Nielsen coincidence number} $N(f:g)$, see Section  \ref{AlgNielsen-section}).

\section{Power maps as torus loops }\label{sec:intersection}

In this section we consider the connections between $n$-valued power maps and loops in the torus. 

For a map $f:X\multimap Y$, let $\Gamma(f)\subset X\times Y$ be the \emph{graph of $f$}, given by 
\begin{equation}
    \Gamma(f) = \{ (x,y) \in X\times Y ;~~ y \in f(x) \}. 
\end{equation}
For a power map $\phi_{n,a}$, we may view the graph $\Gamma(\phi_{n,a})$ as a subset of the two-dimensional torus $\mathbb T^2$. This graph need not be path-connected; however, it can always be decomposed into path-connected components.

For the power map \(\phi_{n,a}\), the branches over the interval
\([0,1]\) may be written as
\begin{equation}
f_j(t)=e^{2\pi i\left(\frac{a}{n}t+\frac{j}{n}\right)},
\qquad j\in\mathbb Z_n.    
\end{equation}
Following the branch \(f_j\) once around the domain circle leads to the
branch \(f_{j+a}\), where the index is taken modulo \(n\), since
$
f_j(1)=f_{j+a}(0)$.
Thus, the path-connected components of \(\Gamma(\phi_{n,a})\)
correspond to the orbits of the translation
$
j\longmapsto j+a \pmod n
$
on $\mathbb Z_n$. We obtain the following lemma.

\begin{lemma}
\label{lemma:pathconnected}
Let \(n\geq 1\) and \(a\in\mathbb Z\). The graph
\(\Gamma(\phi_{n,a})\) has exactly \(\GCD(n,a)\)
path-connected components. In particular,
\(\Gamma(\phi_{n,a})\) is path-connected if and only if
$
\GCD(n,a)=1$.
\end{lemma}



Given an $n$-valued map $f:S^1\multimap S^1$, a \textit{graph splitting} of the graph $\Gamma(f)$ is a decomposition of $\Gamma(f)$ into disjoint sets, each of which is the graph of a $u$-valued map. The number $u$ is given by $u = \frac{n}{k}$, where $k$ is the number of components of $\Gamma(f)$.  The following lemmas state elementary facts about graph splittings, their proofs are straightforward and thus omitted.

\begin{lemma}
\label{lemma:graphsplit}
     Let $\phi_{n,a}$ be a power map and $\Gamma(\phi_{n,a})$ be disconnected. Then there exists a graph splitting into $w = \GCD(n, a)$ different $\frac{n}{w}$-valued maps with degrees equal to $\frac{a}{w}$. 
\end{lemma} 

\begin{lemma}
\label{lemma:splitnumber}
     Let ($\phi_{n,a}, \phi_{m,b}$) be a pair of power maps such that
$am\neq bn$ and let $w = \GCD(n, a)$. Then 
     \begin{equation}
         \# \Coin_{X \times Y} (\phi_{n,a} : \phi_{m,b}) = w \cdot \# \Coin_{X \times Y} (\phi_{\frac{n}{w},\frac{a}{w}} \colon \phi_{m,b}).
     \end{equation} 
\end{lemma}

A loop $l:[0,1]\to X$ is called a \textit{simple loop} if and only if any two different points, except the endpoints 0 and 1, have different images. We denote the image of the loop $l$ by $\Image(l)\subseteq X$. 

A simple loop in the torus may be identified with a path on its covering space $\R^2$.  Every simple loop in the torus at the base point corresponds to a path in $\R^2$ from $(0,0)$ to some point $(x,y) \in \mathbb{Z}^2$.

\begin{definition} \label{def:loopcorr}
Let $\phi_{n,a} \colon S^1 \multimap S^1$ be a power map such that $\GCD(n, a) = 1$. The (simple) loop $l \colon [0,1] \to \mathbb{T}^2$ given by $l(t) = (e^{i2\pi n t}, e^{i2\pi a t}) \in \mathbb{T}^2$ is called the \textit{loop corresponding to the power map} $\phi_{n,a}$. 
\end{definition}

\noindent By Lemma \ref{lemma:pathconnected}, the assumption $\GCD(n, a) = 1$ means that the graph of $\phi_{n,a}$ is path-connected. The loop is represented by the straight path in $\R^2$ between points $(0,0)$ and $(n,a)$, given by $p(t) = (nt,at)$, where $t \in [0,1]$. As the product map $r \times r \colon \R^2 \to \mathbb{T}^2$, where $r \colon \mathbb{R} \to S^1$ is given by $r(t) = e^{i2\pi t}$, is the covering of $\mathbb{T}^2$ by $\R^2$, the points of the loop obtained from the path $p$  are given by $l(t) = (e^{i2\pi n t}, e^{i2\pi a t}) \in \mathbb{T}^2$, where $t \in [0,1]$, and $l(0)=l(1)$. By the assumption $\GCD(n, a)=1$, we get $l(t_1) \neq l(t_2)$ for  distinct $t_1,t_2\in[0,1)$, so the obtained loop is simple. 

\begin{lemma}
\label{lemma:plotisimage}
Let $\phi_{n,a}$ be a power map with $\GCD(n,a)=1$, and let $l$ be its corresponding torus loop. Then $\Image(l) = \Gamma(\phi_{n,a})$.
\end{lemma}

\begin{proof}
    Let $\phi_{n,a}$ be a power map with $\GCD(n, a) = 1$. By Definition~\ref{power} $
    \phi_{n,a}(e^{i2\pi t'}) = \{e^{i2\pi\frac{a}{n}t'},e^{i2\pi(\frac{a}{n}t'+\frac 1n)},\ldots,e^{i2\pi(\frac{a}{n}t'+\frac{n-1}n)}\}$. Let us associate with every $t' \in [0,1]$ the set $T = \{ \frac{t'}{n},  \frac{t'}{n}+\frac 1n, \ldots,  \frac{t'}{n} + \frac{n-1}{n} \}$. Note that the first coordinate of the points in the image $\Image(l)$ is the same for every $t \in T$ i.e. $e^{i2\pi n \frac{t'}{n}} = e^{i2\pi n (\frac{t'}{n} + \frac 1n)} = \ldots = e^{i2\pi n (\frac{t'}{n} + \frac{n-1}{n})}$. 
    
On the other hand, the second coordinates for $t\in T$ form the set
\[
\left\{
e^{i2\pi(\frac{a}{n}t'+\frac{ak}{n})};
\ k=0,\ldots,n-1
\right\}.
\]
Since $\GCD(n,a)=1$, multiplication by $a$ permutes the residue classes
modulo $n$, so this set is equal to
$\phi_{n,a}(e^{i2\pi t'})$.
    
 This means that the image $\Image(l) \subset \mathbb{T}^2$ and the graph $\Gamma(\phi_{n,a}) \subset \mathbb{T}^2$ consist of the same points, and so $\Image(l) = \Gamma(\phi_{n,a})$. 

\end{proof}

\begin{corollary}
\label{cor:intloops}
Let $l_1,l_2$ be two torus loops which are represented by the straight paths $p_1, p_2 \colon [0,1] \to \R^2$ between points $(0,0)$, $(n, a)$ and $(0,0),(m,b)$ respectively. Assume that
$\GCD(n,a)=\GCD(m,b)=1
$
and
$
am\neq bn.
$ Then the number of intersection points of $l_1$ and $l_2$ is given by $\vert am - bn \vert$. (This is consistent with the result in  \cite{Farb}).
\end{corollary}

\begin{proof}
    This follows from Lemma~\ref{lemma:plotisimage} and Theorem~\ref{th:numberofcoinxy}.
\end{proof}

The fundamental group of the torus satisfies
$\pi_1(\mathbb T^2,(1,1))\cong\mathbb Z^2$.
More generally, let $l\colon[0,1]\to\mathbb T^2$ be any loop and let
$\widetilde l\colon[0,1]\to\mathbb R^2$ be any lift of $l$. Then
\begin{equation}
\widetilde l(1)-\widetilde l(0)=(p,q)\in\mathbb Z^2.
\end{equation}
The pair $(p,q)$ determines the homotopy class of the loop. Since
$\pi_1(\mathbb T^2)$ is abelian, it also determines the free homotopy class of $l$.

\begin{lemma}
\label{lemma:homopowerloop1}
For every $n$-valued map $f \colon S^1 \multimap S^1$ with degree equal to $a$, where $\GCD(n, a) = 1$, there exists a simple loop $l_f \colon [0,1] \to \mathbb{T}^2$,  (described by the following formula  (\ref{eq:correspondingloop})) for which $\Gamma(f) = \Image(l_f)$. Moreover, the free homotopy class of $l_f$ corresponds to the element $(n,a)\in\mathbb{Z}^2\cong\pi_1(\mathbb{T}^2)$.

\end{lemma}

\begin{proof}
    The idea of the proof is similar to the proof of Lemma \ref{lemma:plotisimage}. Let $\phi_{n,a} \colon S^1 \multimap S^1$ be the power map and $H \colon S^1 \times [0,1] \multimap S^1$ be a homotopy such that $H(e^{i2\pi t},0) = \phi_{n,a}(e^{i2\pi t})$ and $ H(e^{i2\pi t},1) = f(e^{i2\pi t})$. By Lemma~\ref{lemma:pathconnected} and the assumption $\GCD(n, a) = 1$, the graph of $\phi_{n,a}$ is path-connected. By Lemma~\ref{lemma:plotisimage} the simple loop corresponding to the given power map is given by $l(t) = (e^{i2\pi n t}, e^{i2\pi a t})$. We will construct a homotopy $L \colon [0,1]\times[0,1] \to \mathbb{T}^2$ such that $L(t,0) = l(t)$ and $L(t,1)$ is the desired $l_f(t)$. 
    
    Let us consider $H(e^{i2\pi t}, s)$ for $s \in [0,1]$. For any $t_0 \in [0,1)$ and $s \in [0,1]$ let $G(t_0,s) = H(e^{i2\pi t_0}, s)$.  The map $G \colon [0,1) \times [0,1] \multimap S^1$ has a simply connected domain, so it splits into $n$ different maps $h_0, h_1, \ldots, h_{n-1} \colon [0,1) \times [0,1] \to S^1$. Each of these maps $h_j$ has a unique continuous extension to $h_j \colon [0,1] \times [0,1] \to S^1$. By this construction, for each $j$ we have $h_j(1,s) = h_{(j+a)(\modulo{n})}(0,s)$. We set
    \begin{equation} \label{eq:correspondingloop}
     L(t,s)=
\left(e^{2\pi int},h_{p(j)}(nt-j,s)\right),
\qquad
\frac jn\leq t\leq\frac{j+1}{n},
\end{equation}
where $p(j)\equiv aj\pmod n$ and $j \in \{ 0,1,\ldots, n-1 \}$. It is continuous for every $s \in [0,1]$ because of its definition and the definition of $n$-valued homotopy given by $H(e^{i2\pi t}, s)$. We see that $L(t,0) = l(t)$ and we denote $L(t,1) = l_f(t)$. Clearly $\Gamma(f) = \Gamma(H(e^{i2\pi t},1)) = \{ (e^{i2\pi t}, y) \in \mathbb{T}^2;~~ y \in H(e^{i2\pi t},1) \}$ and $\Image(l_f) = \Image(L(t,1)) = \{ (e^{i2\pi t}, y) \in \mathbb{T}^2;~~ y \in H(e^{i2\pi t},1) \} $ are equal sets. Moreover, $l_f$ is a simple loop. This completes the proof of the existence of a simple loop $l_f$ for which $\Gamma(f) = \Image(l_f)$. Since $L$ is a free homotopy between $l$ and $l_f$, the free homotopy class of $l_f$ is represented by $(n,a)$.
\end{proof}




Using the construction from the proof of Lemma~\ref{lemma:homopowerloop1} we introduce the following definition. 

\begin{definition}
Let $f \colon S^1 \multimap S^1$ be an $n$-valued map with degree equal to $a$, where $\GCD(n, a) = 1$. The (simple) loop $l_f \colon [0,1] \to \mathbb{T}^2$ given by $l_f(t) = L(t,1)$, where the map $L \colon [0,1] \times [0,1] \to \mathbb{T}^2$ is given by formula~(\ref{eq:correspondingloop}), is called the \textit{loop corresponding to the $n$-valued map} $f$.
\end{definition}

\begin{remark}
Although the construction of the corresponding loop $l_f$ involves a
choice of a homotopy $H$ and of the splitting used in the construction,
the resulting loop is well defined up to free homotopy.

Indeed, every loop obtained by
this construction has image $\Gamma(f)$ and represents the same element
$(n,a)\in \pi_1(\mathbb T^2)\cong\mathbb Z^2$. Hence any two such loops are
freely homotopic.

Consequently, the Nielsen intersection indices used below are independent
of these choices, by the homotopy invariance of the Nielsen intersection
index.
\end{remark}

\begin{lemma}
\label{lemma:homotopicifonlyif}
Two $n$-valued maps $f,g\colon S^1\multimap S^1$ having connected graphs are homotopic as $n$-valued maps if and only if their corresponding loops $l_f,l_g\colon [0,1] \to\mathbb T^2$ are freely homotopic.
\end{lemma}

\begin{proof}
Let $a=\deg(f)$ and $b=\deg(g)$. By Lemma~\ref{lemma:homopowerloop1}, the free homotopy classes of $l_f$ and $l_g$ are represented by $(n,a)$ and $(n,b)$, respectively. Since free homotopy classes of loops in $\mathbb T^2$ are classified by elements of $\mathbb Z^2$, the loops $l_f$ and $l_g$ are freely homotopic if and only if $a=b$.

On the other hand, by the homotopy invariance of the degree and Theorem~\ref{th:classification}, the maps $f$ and $g$ are homotopic if and only if $a=b$. Therefore,
\begin{equation*}
f\simeq g
\quad\Longleftrightarrow\quad
a=b
\quad\Longleftrightarrow\quad
l_f\simeq l_g
\quad\text{freely}.\qedhere
\end{equation*}
\end{proof}

\section{The algebraic Nielsen number}\label{AlgNielsen-section}

This section is devoted to introducing the (algebraic) Nielsen coincidence number (Definition \ref{algNielsen} below) for an $(n,m)$-valued pair of maps of the circle. We will make use of the definition of the intersection index given in  \cite{McCord} by McCord. It is defined homologically for maps $f \colon X_1 \to Y$, $g \colon X_2 \to Y$, where $X_1,X_2$ and $Y$ are compact, orientable manifolds of dimensions $p,q$ and $n = p+q$ respectively. The Nielsen intersection index $\Ind(f,g:J)$ is defined for an isolated set of McCord intersections $J \subset \Int(f:g) = \{(x_1,x_2) \in X_1 \times X_2;~~f(x_1) = g(x_2)\}$ (see \cite{McCord} for the details). 

\begin{definition}\label{def:rel}
Two points $(x_1, x_2) $, $(x_3, x_4) \in \Int(f:g)$ are related if there exist paths $\alpha$ in $X_1$ from $x_1$ to $x_3$ and $\beta$ in $X_2$ from $x_2$ to $x_4$ such that $f \circ \alpha$ and $g \circ \beta$ are homotopic relative to the endpoints.
\end{definition}

We would like to adapt McCord's theory to our case, but our set of intersections given by (\ref{graphcoin}) is defined in a slightly different way than  $\Int(f\colon g)$, thus  we first  show the correspondence between both intersection sets in our setting ($X_1=X_2= S^1$, $Y= \mathbb{T}^2$).

Let $(f,g) \colon S^1 \multimap S^1$ be an $(n,m)$-valued pair of maps and $l_f, l_g \colon [0,1] \to \mathbb{T}^2$ be  simple loops corresponding to them. We associate with every graph intersection point $(x,y) \in \Coin_{X \times Y}(f:g)$  the McCord intersection point $(x_1,x_2) \in \Int(l_f:l_g)$ for which 
\begin{equation} \label{formula:correspondence}
    l_f(x_1) = l_g(x_2) = (x,y).
\end{equation}

\begin{lemma}
 Let $(f,g) \colon S^1 \multimap S^1$ be an $(n,m)$-valued pair of maps with degrees equal to $a$ and $b$, respectively, where $\GCD(n, a) = 1$, $\GCD(m, b) = 1$  and let $l_f, l_g \colon [0,1] \to \mathbb{T}^2$ be the simple loops corresponding to them. Two graph intersection points $(x,y)$ and $(\widetilde x, \widetilde y)$ of $(f,g)$ are related if and only if the McCord intersection points $(x_1,x_2)$ and $(x_3,x_4)$ of $(l_f,l_g)$ corresponding to them (by the equality (\ref{formula:correspondence})) are related in the sense of Definition \ref{def:rel}.
\end{lemma}

\begin{proof}
Identifying the parameter circle of each corresponding loop with
$S^1=[0,1]/(0\sim 1)$, we may regard $l_f$ and $l_g$ as maps $l_f,l_g \colon S^1 \to \mathbb T^2$. Since $l_f$ and $l_g$ are simple loops such that $
\Image(l_f)=\Gamma(f)$ and $  \Image(l_g)=\Gamma(g),$ they induce homeomorphisms
\[
\bar l_f \colon S^1 \to \Gamma(f), \qquad
\bar l_g \colon S^1 \to \Gamma(g).
\]
Let $\pr_1,\pr_2 \colon \mathbb T^2=S^1\times S^1 \to S^1$
denote the coordinate projections, and put
\[
\pi_f=\pr_1\circ \bar l_f \colon S^1 \to S^1, \qquad
\pi_g=\pr_1\circ \bar l_g \colon S^1 \to S^1 .
\]
Let us remark the codomain of the projection $\pr_1$ corresponds to space $X$ from definition of graph Nielsen relation and the codomain of the projection $\pr_2$ corresponds to the space $Y$ from the same definition.
Since $f$ and $g$ are locally split, the restrictions
$\pr_1|_{\Gamma(f)}$ and $\pr_1|_{\Gamma(g)}$ are local homeomorphisms.
Hence $\pi_f$ and $\pi_g$ are covering maps.

\smallskip
\noindent{\it ($\Rightarrow$)}
Assume that $(x,y)$ and $(\widetilde x,\widetilde y)$ are related by the graph Nielsen
relation. Then there exists a path $
p \colon [0,1]\to S^1$ with $p(0)=x$, $p(1)=\widetilde x$, such that for the splittings
$f\circ p=\{f_1,\dots,f_n\}$ and $ g\circ p=\{g_1,\dots,g_m\}$
there exist indices $j,k$ such that $
f_j(0)=g_k(0)=y,  f_j(1)=g_k(1)=\widetilde y$
and the paths $f_j, g_k$ are homotopic relative to the endpoints.

Define paths in the graphs by
\[
\sigma_f(t)=(p(t),f_j(t)) \in \Gamma(f), \qquad
\sigma_g(t)=(p(t),g_k(t)) \in \Gamma(g).
\]
Now set
\[
\alpha=\bar l_f^{-1}\circ \sigma_f,\qquad
\beta=\bar l_g^{-1}\circ \sigma_g.
\]
Then $\alpha,\beta\colon [0,1]\to S^1$ are paths satisfying $
\alpha(0)=x_1,~~ \alpha(1)=x_3,~~ 
\beta(0)=x_2$ and $\beta(1)=x_4$
because
\[
\bar l_f(x_1)=(x,y),\ \bar l_f(x_3)=(\widetilde x,\widetilde y),\qquad
\bar l_g(x_2)=(x,y),\ \bar l_g(x_4)=(\widetilde x,\widetilde y).
\]
Moreover,
\[
\bar l_f\circ \alpha=\sigma_f=(p,f_j),\qquad
\bar l_g\circ \beta=\sigma_g=(p,g_k).
\]

Let $M\colon [0,1]\times [0,1]\to S^1$ be a homotopy from $f_j$ to $g_k$ relative to the endpoints. Then
\[
H(t,u)=(p(t),M(t,u))
\]
is a homotopy in $\mathbb T^2$ from $\bar l_f\circ \alpha$ to
$\bar l_g\circ \beta$, relative to the endpoints. Therefore
$(x_1,x_2)$ and $(x_3,x_4)$ are related in the sense of
Definition~\ref{def:rel}.

\smallskip
\noindent{\it ($\Leftarrow$)}
Assume now that $(x_1,x_2)$ and $(x_3,x_4)$ are related in the sense of
Definition~\ref{def:rel}. Then there exist paths $\alpha,\beta \colon [0,1]\to S^1$ such that $
\alpha(0)=x_1,~~ \alpha(1)=x_3,~~
\beta(0)=x_2,~~
\beta(1)=x_4$
and $\bar l_f\circ \alpha$ and $\bar l_g\circ \beta$ are homotopic
relative to the endpoints.

Let $K\colon [0,1]\times [0,1]\to \mathbb T^2$ be such a homotopy, so that
\[
K(t,0)=\bar l_g(\beta(t)),\qquad
K(t,1)=\bar l_f(\alpha(t)).
\]
Projecting to the first coordinate, we obtain a homotopy $
F=\pr_1\circ K \colon [0,1]\times [0,1]\to S^1$
from $q:=\pr_1\circ \bar l_g\circ \beta$
to $p:=\pr_1\circ \bar l_f\circ \alpha$ relative to the endpoints.

\begin{minipage}[t]{.5\textwidth}
\begin{center}
\begin{tikzcd}
& S^1 \arrow[dd, "\pi_g"] \\
& \\
\textcolor{white}{.}[0,1] \arrow[uur, "\beta"] \arrow[r, "q"] & S^1
\end{tikzcd}   
\end{center}
\end{minipage}
\begin{minipage}[t]{.5\textwidth}
\begin{center}
\begin{tikzcd}
& S^1 \arrow[dd, "\pi_g"] \\
& \\
\textcolor{white}{.}[0,1] \times [0,1] \arrow[uur, "\widetilde{F}"] \arrow[r, "F"] & S^1
\end{tikzcd}     
\end{center}

\end{minipage}

Since $\beta$ is a lift of $q$ through the
covering map $\pi_g$, the homotopy lifting property gives a unique lift $\widetilde F \colon [0,1]\times [0,1]\to S^1$ such that
\[
\pi_g\circ \widetilde F = F,\qquad \widetilde F(t,0)=\beta(t).
\]

Define
$\beta'(t)=\widetilde F(t,1).$
Then, for every $t\in[0,1]$, we have
\[
\pi_g(\beta'(t))
=
\pi_g(\widetilde F(t,1))
=
F(t,1)
=
p(t).
\]
Hence $\pi_g\circ \beta' = p.$

Also, $F(0,u)=x$ and $F(1,u)=\widetilde x$ for all $u$, and the fibers
$\pi_g^{-1}(x)$ and $\pi_g^{-1}(\widetilde x)$ are discrete. So the paths
$u\mapsto \widetilde F(0,u)$ and $u\mapsto \widetilde F(1,u)$ are constant, because they are continuous maps into a discrete space.
Hence
\[
\beta'(0)=x_2,\qquad \beta'(1)=x_4.
\]
Moreover, $\beta$ and $\beta'$ are homotopic relative to the endpoints via
$\widetilde F$, so $\bar l_g\circ \beta$ and $\bar l_g\circ \beta'$ are
homotopic relative to the endpoints. Therefore $\bar l_f\circ \alpha \simeq \bar l_g\circ \beta'$ relative to the endpoints. 

Now both paths have the same first coordinate $p$, because
\[
\pr_1\circ \bar l_f\circ \alpha = p
\]
by definition, and
\[
\pr_1\circ \bar l_g\circ \beta' = \pi_g\circ \beta' = p.
\]
Define
\[
f_j=\pr_2\circ \bar l_f\circ \alpha,\qquad
g_k=\pr_2\circ \bar l_g\circ \beta'.
\]
Then
\[
\bar l_f\circ \alpha=(p,f_j),\qquad
\bar l_g\circ \beta'=(p,g_k).
\]
Since $(p(t),f_j(t))\in \Gamma(f)$ for every $t$, we have $f_j(t)\in f(p(t))$,
so $f_j$ is one branch of a splitting of $f\circ p$; similarly, $g_k$ is one
branch of a splitting of $g\circ p$.

Finally, projecting a homotopy between $(p,f_j)$ and $(p,g_k)$ to the second
coordinate gives a homotopy from $f_j$ to $g_k$ relative to the endpoints.
Also $
f_j(0)=y=g_k(0)$ and $ f_j(1)=\widetilde y=g_k(1).$
Thus $(x,y)$ and $(\widetilde x,\widetilde y)$ are related by the graph Nielsen
relation.

This proves the equivalence.
\end{proof}

\begin{corollary}
Let $(f,g)\colon S^1\multimap S^1$ be an $(n,m)$-valued pair of maps with connected graphs, and let $l_f,l_g\colon S^1\to\mathbb{T}^2$ be the corresponding loops. Since $\Image(l_f)=\Gamma(f)$ and $\Image(l_g)=\Gamma(g),$
the loops induce homeomorphisms
\begin{equation*}
\bar l_f\colon S^1\to\Gamma(f), \qquad
\bar l_g\colon S^1\to\Gamma(g).
\end{equation*}
Define $
\Phi\colon \Coin_{X\times Y}(f:g) \to \Int(l_f:l_g)$
by
\begin{equation*}
\Phi(x,y)=
\left(\bar l_f^{-1}(x,y),\bar l_g^{-1}(x,y)\right).
\end{equation*}
Then $\Phi$ is a bijection. Moreover, by the preceding lemma, $\Phi$ preserves the corresponding Nielsen relations and therefore induces a bijection between the graph intersection classes of $(f,g)$ and the McCord intersection classes of $(l_f,l_g)$.
\end{corollary}

\begin{remark}
In the remainder of this section we restrict our attention to pairs
$(f,g)$ for which both graphs $\Gamma(f)$ and $\Gamma(g)$ are connected.
This is the only case needed for the construction below.

The case of disconnected graphs is obtained componentwise. Indeed, writing
\[
\Gamma(f)=\bigsqcup_{i=1}^{r}F_i,
\qquad
\Gamma(g)=\bigsqcup_{j=1}^{s}G_j,
\]
every graph intersection class is contained in $F_i\cap G_j$ for a unique
pair $(i,j)$. Hence the Nielsen number in the general case is obtained by
applying the connected-graph construction to each pair of components and
summing
\[
N(f:g)
=
\sum_{i=1}^{r}\sum_{j=1}^{s} N(f_i:g_j),
\]
where $f_i$ and $g_j$ are the multivalued maps whose graphs are $F_i$ and
$G_j$, respectively.

Thus all arguments below may be carried out for connected graphs without
loss of generality; the general case follows immediately by decomposition
into connected components.
\end{remark}

Now we are ready to define the Nielsen number in the case of maps of the circle.  We use the intersection index given in \cite{McCord} by McCord,
which we briefly recall below. Let $u\colon X_1\to Y$ and $v\colon X_2\to Y$ be maps between compact oriented manifolds such that
\begin{equation*}
\dim X_1+\dim X_2=\dim Y.
\end{equation*}
For an isolated McCord intersection class $J\subset\Int(u:v)$, the index
$\Ind(u,v;J)$ 
is the local intersection number of the map
\begin{equation*}
u\times v\colon X_1\times X_2\longrightarrow Y\times Y
\end{equation*}
with the diagonal $\Delta_Y$, restricted to the intersections belonging to $J$; see \cite{McCord}. In particular, if the intersections are transverse, then $\Ind(u,v;J)$ is the sum of the local intersection signs of the points contained in $J$. Thus, in the case of two oriented loops in the oriented torus, it is the algebraic, or signed, number of their intersection points belonging to the class $J$.

\begin{definition}\label{algNielsen}
Let $C_0\subset\Coin_{X\times Y}(f:g)$ be a graph intersection class and let $C_0'=\Phi(C_0)\subset\Int(l_f:l_g)$
be the corresponding McCord intersection class. The class $C_0$ is called \textit{(algebraically) essential} if $\Ind(l_f,l_g;C_0')\neq 0$
and \textit{(algebraically) inessential} otherwise. The \textit{Nielsen coincidence number} $N(f:g)$ is the number of algebraically essential graph intersection classes.
\end{definition}

Note that $N(f:g)$  is a homotopy invariant because of the homotopy invariance of the Nielsen intersection index.
Since the McCord intersection index is homotopy invariant and the empty
class has index equal to zero, every algebraically essential graph intersection class
is geometrically essential. Hence
\begin{equation}\label{inequal}
 N(f:g)\leq \hat N(f:g).
\end{equation}

\section{Main result}\label{sec:main}

We are now in a position to prove the main result of this paper.

\begin{theorem}\label{main}
Let $(f,g)$ be an $(n,m)$-valued pair of self-maps of $S^1$ with degrees equal to $a, b$ respectively. Then
\begin{equation}
     N(f:g) = \vert am - bn \vert.
\end{equation}
\end{theorem}

\begin{proof}
Assume first that
$am \neq bn, \GCD(n, a)=1$ and $\GCD(m, b)=1.$
By the Classification Theorem~\ref{th:classification}, we have
$f \simeq \phi_{n,a}$ and $g \simeq \phi_{m,b}.$
Since the Nielsen coincidence number is a homotopy invariant, it is enough to compute $N(\phi_{n,a}:\phi_{m,b})$.

By Lemma~\ref{lemma:plotisimage}, the graphs of $\phi_{n,a}$ and $\phi_{m,b}$ are the images of the corresponding simple loops
$
l_{n,a},\, l_{m,b}\colon [0,1]\to \mathbb T^2.
$
Moreover, by Lemma~\ref{lemma:difclass}, every graph intersection point of the pair $(\phi_{n,a},\phi_{m,b})
$ forms a separate graph intersection class. Hence every graph intersection class consists of a single point. By the correspondence between graph intersection classes and McCord intersection classes established in Section~\ref{AlgNielsen-section}, each corresponding McCord intersection class also consists of a single point.

Let $C_0' \subset \Int(l_{n,a}:l_{m,b})$ be such a McCord intersection class. Since $am\neq bn$, the loops $l_{n,a}$ and $l_{m,b}$ are represented on the universal covering $\mathbb R^2$ by straight lines with distinct direction vectors $(n,a)$ and $(m,b)$. Therefore every intersection is transversal, and the local intersection sign is the same at every intersection point. 
Indeed, it is determined by the sign of
\[
\det
\begin{pmatrix}
n & m\\
a & b
\end{pmatrix}
= nb-am.
\]
Consequently,
\[
\Ind(l_{n,a}:l_{m,b};C_0')=\pm 1 \neq 0.
\]
Thus every graph intersection class of $(\phi_{n,a},\phi_{m,b})$ is algebraically essential.

It follows that $N(\phi_{n,a}:\phi_{m,b})$
is equal to the number of graph intersection classes of $(\phi_{n,a},\phi_{m,b})$. Since each such class consists of exactly one point, by Theorem~\ref{th:numberofcoinxy} we obtain
\[
N(\phi_{n,a}:\phi_{m,b})
=
\# \Coin_{X\times Y}(\phi_{n,a}:\phi_{m,b})
=
|am-bn|.
\]
Finally, by homotopy invariance,
\[
N(f:g)=N(\phi_{n,a}:\phi_{m,b})=|am-bn|.
\]

If $n,a$ (or/and $m,b$) are not relatively prime, put
$r=\GCD(n,a)$, $s=\GCD(m,b)$.
By Lemma~\ref{lemma:graphsplit}, the graphs split into $r$ and $s$
connected components of types
$
\phi_{\frac nr,\frac ar}$
and $
\phi_{\frac ms,\frac bs}$,
respectively. We consider the resulting $rs$ pairs of connected components
separately. By the connected-graph case and the componentwise definition of
the Nielsen number, we obtain
\[
N(\phi_{n,a}:\phi_{m,b})=|am-bn|.
\]


If $am=bn$, choose $\varepsilon>0$ as in
Lemma~\ref{lemma:domainS1}. By homotopy invariance,
\[
N(f:g)
=
N(\phi_{n,a}:\phi_{m,b}+\varepsilon)=0,
\]
since the shifted pair has no graph intersections. This completes the proof.

\end{proof}

\begin{theorem}\label{problem1}
Let $(f,g)$ be an $(n,m)$-valued pair of self-maps of $S^1$ with degrees equal to $a, b$ respectively. Then
\begin{equation}
    \hat N(f:g) = \vert am - bn \vert.
\end{equation}
\end{theorem}

\begin{proof}
Assume first that $am \neq bn$. By the Classification Theorem~\ref{th:classification}, we have
$f \simeq \phi_{n,a}$ and $g \simeq \phi_{m,b}$. Since the Nielsen coincidence number is a homotopy invariant $\hat N(f:g) = \hat N(\phi_{n,a}:\phi_{m,b}).
$
From Theorem~\ref{th:numberofcoinxy} 
\[
\# \Coin_{X \times Y}(\phi_{n,a}:\phi_{m,b}) = \vert am - bn \vert
\]
and from Lemma~\ref{lemma:difclass} every graph intersection point is in a different graph intersection class. So we conclude 
\begin{equation}\label{eq:Ngeomfromthetop}
\hat N(\phi_{n,a}:\phi_{m,b}) \leqslant \vert am - bn \vert.
\end{equation}
From the main Theorem~\ref{main} 
\[
N(\phi_{n,a}:\phi_{m,b}) = \vert am - bn \vert,
\]
and every one-pointed graph intersection class of the pair $(\phi_{n,a},\phi_{m,b})$ is algebraically essential. As the intersection index is preserved under a homotopy (\cite{McCord}, Section 2, Property 1), every algebraically essential class must also be geometrically essential. So
\begin{equation}\label{eq:Ngeomfromthebottom}
    \hat N(\phi_{n,a}:\phi_{m,b}) \geqslant N(\phi_{n,a}:\phi_{m,b}) = \vert am - bn \vert.
\end{equation}
Combining equations~\eqref{eq:Ngeomfromthetop} and \eqref{eq:Ngeomfromthebottom}, and from the homotopy invariance of the geometric Nielsen coincidence number, we obtain the result 
\[
\hat N(f:g) = \hat N(\phi_{n,a}:\phi_{m,b}) = \vert am - bn \vert.
\]

If $am = bn$ we obtain easily a pair of maps with no graph intersections as in Proposition~\ref{prop:N=0}, so $ \hat N(f : g) = 0$ which completes the proof.
\end{proof}

 \begin{corollary}
Let $(f,g)$ be an $(n,m)$-valued pair of self-maps of $S^1$ with
degrees $a$ and $b$, respectively. Then
\begin{equation}
\MC_{X\times Y}(f:g)= \vert am-bn \vert.  
\end{equation}
More precisely, if $am\neq bn$, then the minimum is realized by the pair
of power maps $(\phi_{n,a},\phi_{m,b})$, and
\[
\MC_{X\times Y}(f:g)
=
\#\Coin_{X\times Y}(\phi_{n,a}:\phi_{m,b})
=
|am-bn|.
\]
If $am=bn$, then
\[
\MC_{X\times Y}(f:g)=0,
\]
and this minimum is realized by the shifted pair $(\phi_{n,a},\phi_{m,b}+\varepsilon)$,
where $\varepsilon>0$ is chosen as in Lemma~\ref{lemma:domainS1}.
\end{corollary}

We can view the above as a Wecken-type property for $S^1$ for $(n,m)$-valued pairs of maps.
\begin{corollary}
For every $(n,m)$-valued pair of maps $f,g \colon S^1 \multimap S^1$, there exists a homotopic pair $\widetilde f,\widetilde g \colon S^1 \multimap S^1$ such that $\# \Coin_{X \times Y} (\widetilde f: \widetilde g) =  N(f:g) = \hat N(f:g)$.
\end{corollary}

To give a method to estimate the minimal number of domain coincidence points remains an open problem.
\begin{problem}
Let $am\neq bn$, and let $(f,g)$ be an $(n,m)$-valued pair of self-maps of
$S^1$ with degrees $a$ and $b$, respectively. In general,
\[
\MC_X(f:g)\leq \MC_{X\times Y}(f:g),
\]
and these two quantities need not be equal, as illustrated in
Figure~\ref{introfig}. Is there a method for computing
\[
\MC_X(f:g)
\]
for an arbitrary $(n,m)$-valued pair of self-maps of $S^1$?
\end{problem}

\newpage

\bibliographystyle{amsplain}

\end{document}